\pdfoutput=1
\documentclass[reqno,11pt]{amsart}
\usepackage{geometry}
\usepackage{versions}

\usepackage[T1]{fontenc}
\usepackage{type1ec}
\usepackage[utf8]{inputenc}
\usepackage[english]{babel}
\usepackage{csquotes}
\usepackage[shortcuts]{extdash} 

\usepackage[usenames,dvipsnames]{xcolor} 
\usepackage[
 unicode=true,
 pdfusetitle,
 bookmarks=true,
 bookmarksnumbered=false,
 bookmarksopen=false,
 breaklinks=false,
 pdfborder={0 0 1},
 backref=false,
 colorlinks=true,
 linkcolor=blue,
 citecolor=MidnightBlue,
 urlcolor=Violet,
 hypertexnames=false,
 ]
 {hyperref}

\usepackage{amsmath}
\usepackage{amssymb}
\usepackage{amsthm}
\usepackage{thmtools}
\usepackage{mathtools}
\usepackage{esint}
\usepackage{tikz-cd}

\usepackage[
 activate={true,nocompatibility},
 final,
 tracking=true,
 kerning=true,
 spacing=true,
 factor=1100,
 stretch=10,
 shrink=10
]{microtype}
\SetTracking{encoding={*}, shape=sc}{0} 
\microtypecontext{spacing=nonfrench}

\usepackage{lmodern}
\usepackage{mathrsfs} 
\usepackage[mathcal]{euscript} 

\usepackage[square,numbers]{natbib}

\numberwithin{equation}{section}
\numberwithin{figure}{section}

\declaretheorem[name=Theorem,style=plain,numberwithin=section]{thm}
\newtheorem*{thm*}{Theorem}

\declaretheorem[name=Remark,style=remark,sibling=thm]{rem}

\usepackage{cleveref}
\crefname{thm}{Theorem}{Theorems}
\crefname{prop}{Proposition}{Propositions}
\crefname{lem}{Lemma}{Lemmas}
\crefname{cor}{Corollary}{Corollaries}
\crefname{example}{Example}{Examples}
\crefname{defn}{Definition}{Definitions}
\crefname{rem}{Remark}{Remarks}
\crefname{claim}{Claim}{Claims}
\crefname{assum}{Assumption}{Assumptions}
\crefname{enumi}{}{}
\crefname{enumii}{}{}
\crefname{enumiii}{}{}
\crefname{equation}{}{}

\usepackage[inline]{enumitem}
\newcommand{\enumlabelformat}{\roman}
\newcommand{\enumlabelfont}[1]{#1}
\newlength{\thelabelsep}
\setlist{labelsep=\thelabelsep}
\setlist[enumerate]{font=\enumlabelfont,label=(\enumlabelformat*),leftmargin=2.5em}
\setlist[itemize]{leftmargin=2.5em,label=$-$}

\newcounter{inlineenum}
\renewcommand{\theinlineenum}{\enumlabelformat{inlineenum}}
\newenvironment{inlineenum}
  {\setcounter{inlineenum}{0}%
   \renewcommand{\item}{\refstepcounter{inlineenum}{(\theinlineenum)\hspace{\thelabelsep}}}
  }
  {\ignorespacesafterend}

\newcommand\numberthis{\addtocounter{equation}{1}\tag{\theequation}}

\newcommand{\overbar}[1]{\mkern 1mu\overline{\mkern-1mu#1\mkern-1mu}\mkern 1mu} 
\newcommand{\ol}[1]{\overline{#1}}

\newcommand{\R}{\mathbb{R}}
\newcommand{\Cplx}{\mathbb{C}}

\newcommand{\img}[1]{\operatorname{img}(#1)} 
\DeclareMathOperator{\dimsymb}{dim}
\renewcommand{\dim}[2][]{\dimsymb_{#1}(#2)} 
\newcommand{\trace}{\operatorname{tr}}
\newcommand{\spec}[2][]{{\sigma_{#1}(#2)}} 
\newcommand{\essspec}[1]{{\sigma_e(#1)}} 
\newcommand{\dom}[1]{\operatorname{dom}(#1)} 

\newcommand{\bergman}[1][2]{A^{#1}} 
\newcommand{\dbar}{{\smash{\overbar{\partial}}}} 
\newcommand{\htensor}{\mathbin{\hat\otimes}} 

\DeclareMathOperator*{\psum}{
 \mathchoice
  {\sideset{}{^{\,\prime}}{\sum}}
  {\sum^{\prime}\!\!}
  {\sum^{\prime}\!\!}
  {\sum^{\prime}\!\!}
}

\newlength{\hilbcomparrow}
\usepackage{xparse}
\NewDocumentCommand{\dbarN}     { O{} O{} }{N_{#1}^{#2}}       
\NewDocumentCommand{\clapl} { O{} O{} }{\square_{#1}^{#2}} 
\NewDocumentCommand{\minsol}{ O{} O{} }{S_{#1}^{#2}}       

\newcommand{\ovli}{\overline}
\newcommand{\ovprt}{\dbar}

\begin{document}

\bibliographystyle{abbrvnatnourls}

\title{On some spectral properties of the weighted \texorpdfstring{$\dbar$}{dbar}-Neumann operator}
\author{Franz Berger}
\author{Friedrich Haslinger}
\address{Fakult\"at f\"ur Mathematik, Universit\"at Wien, Oskar-Morgenstern-Platz 1, A-1090 Wien, Austria}
\email{franz.berger2@univie.ac.at, friedrich.haslinger@univie.ac.at}
\thanks{Partially supported by the Austrian Science Fund (FWF): P23664.}

\begin{abstract}
 We study necessary conditions for compactness of the weighted $\dbar$-Neumann operator on the space $L^2(\mathbb C^n,e^{-\varphi})$ for a plurisubharmonic function $\varphi$.
 Under the assumption that the corresponding weighted Bergman space of entire functions has infinite dimension, a weaker result is obtained by simpler methods.
 Moreover, we investigate (non-) compactness of the $\dbar$-Neumann operator for \emph{decoupled} weights, which are of the form $\varphi(z) = \varphi_1(z_1) + \dotsb + \varphi_n(z_n)$.
 More can be said if every $\Delta\varphi_j$ defines a nontrivial doubling measure.
\end{abstract}


\maketitle
\tableofcontents

\section{Introduction}

We consider the Cauchy-Riemann complex (or Dolbeault complex)
$$ 0 \to C^\infty(\mathbb C^n) \xrightarrow{\dbar} \Omega^{0,1}(\mathbb C^n) \xrightarrow{\dbar} \Omega^{0,2}(\mathbb C^n) \xrightarrow{\dbar} \cdots \xrightarrow{\dbar}\Omega^{0,n}(\mathbb C^n) \to 0, $$
where $\Omega^{0,q}(\mathbb C^n) \coloneqq C^\infty(\mathbb C^n,\Lambda^{0,q}T^*(\mathbb C^n))$ denotes the space of smooth $(0,q)$-forms on $\mathbb C^n$, and the operator $\dbar$ is defined by
$$ \dbar u \coloneqq \psum_{|J|=q} \sum_{j=1}^n \frac{\partial u_J}{\partial \overline z_j}\, d\overline z_j \wedge d\overline z_J $$
for $u=\psum_{|J|=q}u_J\,d\overline{z}_J \in \Omega^{0,q}(\mathbb C^n)$, and the primed sum indicates that summation is done only over increasing multiindices $J=(j_1, \dots , j_q)$ of length $q$, and $d\overline z_J \coloneqq d\overline z_{j_1} \wedge \dots \wedge d\overline z_{j_q}$.

If $\varphi \colon \mathbb{C}^n \to \mathbb{R}$ is a $C^2$ function, we denote by $L^2_{0,q}(\mathbb C^n,e^{-\varphi})$ the Hilbert space completion of $\Omega^{0,q}_c(\mathbb C^n)$ (compactly supported elements of $\Omega^{0,q}(\mathbb C^n)$) with respect to the inner product
\begin{equation}\label{eq:weighted_inner_prod}
 (u,v)_\varphi \coloneqq \psum_{|J|=q} \int_{\mathbb C^n} u_J\overline{v_J}\,e^{-\varphi} d\lambda,
\end{equation}
with $\lambda$ the Lebesgue measure on $\mathbb C^n$.
Let $\|\cdot\|_\varphi$ be the associated norm.
The $\dbar$-operator extends maximally to a closed operator on $L^2_{0,q}(\mathbb C^n,e^{-\varphi})$ by letting it act in the sense of distributions on
$$ \dom{\dbar} \coloneqq \big\{ f \in L^2_{0,q}(\mathbb C^n, e^{-\varphi}) : \dbar f \in L^2_{0,q+1}(\mathbb C^n, e^{-\varphi})\big\}. $$
In other words, $\dbar f$ is defined to be the distributional derivative of $f$, viewed as an element of $L^2_{0,q+1}(\mathbb C^n, e^{-\varphi})$.
Note that the symbol of its formal adjoint with respect to \cref{eq:weighted_inner_prod}, which we denote by $\dbar_\varphi^t$, depends on the weight $\varphi$, and we denote the Hilbert space adjoint of $\dbar$ by $\dbar_\varphi^*$.
For more details on the weighted $\dbar$-problem, see \cite{Haslinger2007,Haslinger2014}.

The \emph{complex Laplacian} is the Laplacian of the $\dbar$-complex,
\begin{equation}\label{eq:complex_laplacian_smooth}
 \clapl[0,q][\varphi] \coloneqq \dbar \dbar_\varphi^t + \dbar_\varphi^t \dbar \colon \Omega^{0,q}(\mathbb C^n) \to \Omega^{0,q}(\mathbb C^n).
\end{equation}
This is an elliptic, formally self-adjoint differential operator, and its \emph{Gaffney extension}, which we also denote by $\clapl[0,q][\varphi]$, is then
$$ \clapl[0,q][\varphi] \coloneqq \dbar\dbar_\varphi^* + \dbar_\varphi^*\dbar, $$
understood as an unbounded operator on $L^2_{0,q}(\mathbb C^n,e^{-\varphi})$ with domain
$$ \dom{\clapl[0,q][\varphi]} \coloneqq \big\{ u \in \dom{\dbar} \cap \dom{\dbar^*_\varphi}: \dbar u \in  \dom{\dbar^*_\varphi}\text{ and } \dbar^*_\varphi u \in  \dom{\dbar}\big\}, $$
where $\dbar^*_\varphi$ is the Hilbert space adjoint of $\dbar$ for the inner product~\cref{eq:weighted_inner_prod}, and $\dom{\dbar^*_\varphi}$ is its domain.
It is a positive self-adjoint operator on $L^2_{0,q}(\mathbb C^n,e^{-\varphi})$ by techniques going back to \cite{Gaffney1955}, see also \cite{Haslinger2014}.
By completeness of $\mathbb C^n$, \cref{eq:complex_laplacian_smooth} is essentially self-adjoint, so that the Gaffney extension is its only closed symmetric extension, see for instance \cite{Haslinger2014,Ma2007}

The inverse of the restriction of $\clapl[0,q][\varphi]$ to $\dom{\clapl[0,q][\varphi]}\cap\ker(\clapl[0,q][\varphi])^\perp$ is called the \emph{$\dbar$-Neumann operator} and is denoted by
$$ \dbarN[0,q][\varphi] \coloneqq \big(\clapl[0,q][\varphi]|_{\dom{\clapl[0,q][\varphi]}\cap\ker(\clapl[0,q][\varphi])^\perp}\big)^{-1} \colon \img{\clapl[0,q][\varphi]} \to L^2_{0,q}(\mathbb C^n,e^{-\varphi}). $$
While $\dbarN[0,q][\varphi]$ is initially only defined on $\img{\clapl[0,q][\varphi]}$, we can extend it to a bounded operator on $L^2_{0,q}(\mathbb C^n,e^{-\varphi})$ if and only if $\clapl[0,q][\varphi]$ has closed range.
This is equivalent to $\img{\dbar}\cap L^2_{0,q}(\mathbb C^n,e^{-\varphi})$ and $\img{\dbar}\cap L^2_{0,q+1}(\mathbb C^n,e^{-\varphi})$ both being closed, and also to the existence of $C > 0$ such that
\begin{equation}\label{eq:basic_estimate}
 \|u\|_\varphi^2 \leq C \big(\|\dbar u\|_\varphi^2 + \|\dbar^*_\varphi u\|_\varphi^2\big)
\end{equation}
holds for all $u \in \dom{\dbar} \cap \dom{\dbar^*_\varphi} \cap \ker(\clapl[0,q][\varphi])^\perp$, see the arguments in \cite{Hoermander1965,Straube2010,Berger2015}.
We are interested in spectral properties of $\dbarN[0,q][\varphi]$, in particular the problem of deciding its compactness in terms of easily accessible properties of $\varphi$.
We denote by $\spec{\clapl[0,q][\varphi]}$ and $\essspec{\clapl[0,q][\varphi]}$ the \emph{spectrum} and the \emph{essential spectrum} of $\clapl[0,q][\varphi]$, respectively.
The latter set consists of all points in $\spec{\clapl[0,q][\varphi]}$ which are accumulation points of the spectrum or eigenvalues of infinite multiplicity.
If $\dbarN[0,q][\varphi]$ is a bounded operator, then it follows from general operator theory that its compactness is equivalent to $\essspec{\clapl[0,q][\varphi]} \subseteq \{0\}$.
Equivalently, the spectrum of $\clapl[0,q][\varphi]$ may only consist of discrete eigenvalues of finite multiplicity, except possibly allowing for an infinite dimensional kernel.

One important property of the $\dbar$-Neumann operator lies in the fact that it can be used to compute the norm-minimal solution to the inhomogeneous equation $\dbar u = f$, for a given $\dbar$-closed $(0,q+1)$-form $f$.
Indeed, this solution is then given by $u = \dbar^*_\varphi \dbarN[0,q+1]f$, and many operator theoretic properties such as compactness transfer from $\dbarN$ to the minimal solution operator $\dbar^*_\varphi\dbarN$.
We refer the reader to the literature, for example \cite{Straube2010,Haslinger2014}.

An important tool is the (weighted) \emph{Kohn-Morrey-H\"ormander formula},
\begin{equation}\label{eq:kohn_morrey}
 \|\dbar u\|_\varphi^2 + \|\dbar_\varphi^* u\|_\varphi^2
  = \psum_{|J|=q} \, \sum_{j=1}^n \bigg\|\frac{\partial u_J}{\partial \overline z_j}\bigg\|_\varphi^2
  + \psum_{|K|=q-1} \,\sum_{j,k=1}^n \int_{\mathbb C^n} \frac{\partial^2\varphi}{\partial z_j \partial \overline z_k}\,u_{jK} \overline{u_{kK}}\,e^{-\varphi}d\lambda,
\end{equation}
see \cite{Hoermander1965}, \cite[p.~109]{Haslinger2014}, or \cite[Proposition~2.4]{Straube2010}, which is valid a~priori for $u \in \Omega^{0,q}_c(\mathbb C^n)$ with $1 \leq q \leq n$, and by extension also for $u \in \dom{\dbar} \cap \dom{\dbar_\varphi^*}$, since $\Omega^{0,q}_c(\mathbb C^n)$ is dense in the latter space for the norm $u \mapsto (\|u\|_\varphi^2 + \|\dbar u\|_\varphi^2 + \|\dbar_\varphi^* u\|_\varphi^2)^{1/2}$.
Again, the reason for this is that $\mathbb C^n$ is complete for the Euclidean metric.

\begin{rem}\label{rem:introduction}
 \begin{inlineenum}
  \item
   From \cref{eq:kohn_morrey}, one derives a sufficient condition for boundedness of $\dbarN[0,q][\varphi]$, which is that $\varphi$ be plurisubharmonic, and
   $$ \liminf_{z \to \infty} s_q(z) > 0, $$
   where $s_q(z)$ is the sum of the $q$ smallest eigenvalues of the complex Hessian
   $$ M_\varphi(z) \coloneqq \bigg(\frac{\partial^2\varphi}{\partial z_j\partial\overline z_k}(z)\bigg)_{j,k=1}^n $$
   of $\varphi$.
   W refer to \cite{Haslinger2014} for the details.
   From this condition it also follows that $\clapl[0,q][\varphi]$ is injective, so that $\dbarN[0,q][\varphi]$ is the bounded inverse of $\clapl[0,q][\varphi]$.%
   \footnote{Actually, for the Laplacian to be injective it suffices to have $\varphi$ plurisubharmonic and $M_\varphi(z) \neq 0$ for some $z\in \mathbb C^n$, compare with the arguments in the proof of \cref{dbar_neumann_noncompactness_decoupled_weighted_spaces}.}
  \item
   A sufficient condition for $\dbarN[0,q][\varphi]$ being compact is the following:
   Suppose that $\varphi \in C^2(\mathbb C^n, \mathbb R)$ and that the sum $s_q(z)$ of the smallest $q$ eigenvalues of $M_\varphi(z)$ has the property
   \begin{equation} \label{sq}
    \lim_{z \to \infty} s_q(z)= +\infty.
   \end{equation}
   Then $\dbarN[0,q][\varphi]$ is compact.
   Indeed, \cref{eq:kohn_morrey} together with some linear algebra implies that $\|\dbar u\|^2_\varphi + \|\dbar^*_\varphi u\|^2_\varphi \geq (s_q u,v)_\varphi$ for all $u,v \in \dom{\dbar}\cap\dom{\dbar^*_\varphi} \cap L^2_{0,q}(\mathbb C^n,e^{-\varphi})$, and hence the diverging of $s_q$ implies compactness, see \cite{Haslinger2011,Haslinger2014} or also \cite{Ruppenthal2011a} for the details.
   Further recent results concerning compactness of the $\dbar$-Neumann operator for the weighted problem can be found in \cite{DallAra2014,DallAra2015}.
   
  \item
   Boundedness and compactness of the $\dbar$-Neumann operator ``percolate'' up the $\dbar$-complex: if $\dbarN[0,q][\varphi]$ is bounded or compact, then $\dbarN[0,q+1][\varphi]$ has the same property, see \cite{Haslinger2014}.
 \end{inlineenum}
\end{rem}

\section{Overview of the results}

In this article we will first derive a necessary condition for compactness to hold, which is weaker than \cref{sq}:
Under the assumption that the \emph{weighted Bergman space}
$$ \bergman(\mathbb C^n, e^{-\varphi}) \coloneqq \bigg\{ f\colon\mathbb C^n \to \mathbb C : \text{$f$ entire and } \int_{\mathbb C^n} |f|^2\,e^{-\varphi}d\lambda < \infty \bigg\} $$
has infinite dimension (see \cref{rem:necessary_condition} for a known sufficient condition for this) and $\dbarN[0,q][\varphi]$ is compact for some $1 \leq q \leq n$, then
\begin{equation}\label{eq:result_limsup_trace_at_infty}
 \limsup_{z \to \infty} \trace(M_\varphi (z))=+\infty,
\end{equation}
or, in other words, $\trace(M_\varphi)$ is unbounded.
A stronger result, obtained with different methods, is presented in \cref{sec:schrodinger}:
it says that if $\dbarN[0,q][\varphi]$ is compact (regardless of dimensionality of the Bergman space), then
$$ \lim_{z \to \infty} \int_{B_1(z)} \trace(M_\varphi)^2\,d\lambda = +\infty. $$
Of course, this also implies~\cref{eq:result_limsup_trace_at_infty}.

In \cref{sec:decoupled_weights} we will consider weights which are \emph{decoupled}, i.e., of the form
$$ \varphi (z_1,\dotsc , z_n) = \varphi_1(z_1) + \dotsb + \varphi_n(z_n). $$
Decoupled weights are known from \cite{Haslinger2007} to be an obstruction to compactness of the $\dbar$-Neumann operator on $(0,1)$-forms, provided at least one $\varphi_j$ gives rise to an infinite dimensional weighted Bergman space (of entire functions on $\mathbb C$).
By using results of \cite{Berger2015,Marco2003,Marzo2009}, we will characterize compactness of $\dbarN[0,q][\varphi]$ for this class of weights, under the additional assumption that each $\varphi_j : \mathbb C \to \mathbb R$ is a subharmonic function such that $\Delta\varphi_j$ defines a nontrivial doubling measure.
This means that $\varphi_j$ is not harmonic and there is $C > 0$ such that 
$$\int_{B_{2r}(z)} \Delta\varphi_j\,d\lambda \leq C \int_{B_r(z)} \Delta\varphi_j\,d\lambda$$
for all $z\in \mathbb C$ and $r > 0$.
As an example, $\varphi(z) \coloneqq |z|^\alpha$ for $\alpha>0$ satisfies these conditions.
It turns out that $\dbarN[0,q][\varphi]$ is never compact for $0 \leq q \leq n-1$, and compactness of $\dbarN[0,n][\varphi]$ depends on the asymptotics of the integral $\int_{B_1(z)}\trace(M_\varphi)\,d\lambda$ as $z \to \infty$.
This will also give examples of weights $\varphi$ such that \cref{eq:result_limsup_trace_at_infty} is not sufficient for compactness of $\dbarN[][\varphi]$.


\section{A necessary condition for compactness}

To derive a necessary condition for compactness of $\dbarN[0,q][\varphi]$, we will apply $\clapl[0,q][\varphi]$ to $(0,q)$-forms with coefficients (for the standard basis) belonging to $\bergman(\mathbb C^n, e^{-\varphi})$.
The point is that $\clapl[0,q][\varphi]$ restricts to a multiplication operator on this space.

\begin{thm}\label{suffcompgen}
 Let $\varphi \colon \mathbb C^n \to \mathbb R$ be a plurisubharmonic $C^2$ function and suppose that the corresponding weighted space $ \bergman(\mathbb C^n, e^{-\varphi})$ of entire function is infinite dimensional. If there is $1 \leq q \leq n$ such that the $\dbar$-Neumann operator $\dbarN[0,q][\varphi]$ is compact, then 
 \begin{equation}\label{comptrgen}
  \limsup_{z \to \infty}\trace(M_\varphi (z)) = +\infty.
 \end{equation}
\end{thm}

\begin{proof}
 Let $g$ be a smooth $(0,q)$-form, $g = \psum_{|J|=q} g_J\,d\overline z_J  \in  \Omega^{0,q}(\mathbb C^n)$.
 If the coefficients $g_J$ of $g$ are holomorphic, then $g \in \ker(\dbar \colon \Omega^{0,q}(\mathbb C^n) \to \Omega^{0,q+1}(\mathbb C^n))$, and $\clapl[0,q][\varphi]g$ as in \cref{eq:complex_laplacian_smooth} is given by
 \begin{align*}
 \dbar \dbar_\varphi^t g & = \dbar \bigg( -\psum_{|K|=q-1}\, \sum_{k=1}^n \bigg(\frac{\partial}{\partial z_k}- \frac{\partial \varphi}{\partial z_k} \bigg) g_{kK}\, d\overline z_K \bigg)\\
   & = - \sum_{j=1}^n \, \psum_{|K|=q-1}\, \sum_{k=1}^n
          \bigg( \frac{\partial^2}{\partial \overline z_j \partial z_k} -\frac{\partial \varphi}{\partial z_k}
          \frac{\partial}{\partial \overline z_j} - \frac{\partial^2 \varphi}{\partial \overline z_j \partial z_k}
          \bigg) g_{kK}\, d\overline z_j \wedge d\overline z_K,
 \end{align*}
 which reduces to
 \begin{equation}\label{compgen}
  \psum_{|K|=q-1} \, \sum_{j,k=1}^n  \frac{\partial^2 \varphi}{\partial \overline z_j \partial z_k} g_{kK}\, d\overline z_j \wedge d\overline z_K,
 \end{equation}
 see \cite{Haslinger2014} for the computation of $\dbar_\varphi^t$.
 If $\widetilde\square^\varphi_{0,q}$ denotes the maximal closed extension of \cref{eq:complex_laplacian_smooth} to an operator on $L^2_{0,q}(\mathbb C^n,e^{-\varphi})$, defined in the sense of distributions on 
 $$ \dom{\widetilde\square^\varphi_{0,q}} \coloneqq \big\{u \in L^2_{0,q}(\mathbb C^n,e^{-\varphi}) : \clapl[0,q][\varphi]u \in L^2_{0,q}(\mathbb C^n,e^{-\varphi})\big\}, $$
 then $\widetilde\square^\varphi_{0,q}$ is symmetric and hence must equal the Gaffney extension.

 Now assume that \cref{comptrgen} is wrong, so that there exists $C\geq 0$ such that
 \begin{equation}\label{contrary}
  \trace(M_\varphi (z)) \le C
 \end{equation}
 for all $z \in \mathbb C^n$.
 Since $\varphi$ is assumed to be plurisubharmonic, $M_\varphi$ is nonnegative and hence \cref{contrary} is equivalent to all entries of $M_\varphi$ being bounded.
 Indeed, $\trace(M_\varphi^* M_\varphi) = \trace(M_\varphi^2) \leq \trace(M_\varphi)^2$, and the Hilbert-Schmidt norm $A \mapsto \trace(A^*A)^{1/2}$ on the space of complex $n\times n$ matrices is equivalent to the max norm (i.e., the maximum of the entries).
 It now follows from \cref{compgen} that all $(0,q)$-forms $g$ with coefficients in $\bergman(\mathbb C^n, e^{-\varphi})$ belong to $\dom{\widetilde\square^\varphi_{0,q}} = \dom{\clapl[0,q][\varphi]}$, and
 $$ \clapl[0,q][\varphi] g = \psum_{|K|=q-1} \, \sum_{j,k=1}^n  \frac{\partial^2 \varphi}{\partial \overline z_j \partial z_k}
 g_{kK}\, d\overline z_j \wedge d\overline z_K.$$
 Hence, if all such $g$ belong to a bounded set of $L^2_{(0,q)}( \mathbb C^n, e^{-\varphi}),$ the image under $\clapl[0,q][\varphi] $ of this bounded set will be again bounded in $L^2_{(0,q)}( \mathbb C^n, e^{-\varphi}).$ In addition we have 
 $$ \dbarN[0,q][\varphi]\Bigg(\psum_{|K|=q-1} \, \sum_{j,k=1}^n  \frac{\partial^2 \varphi}{\partial \overline z_j \partial z_k}
 g_{kK}\, d\overline z_j \wedge d\overline z_K\Bigg) = g, $$
 so that the identity on the (infinite dimensional) space of $(0,q)$-forms with coefficients in $\bergman(\mathbb C^n,e^{-\varphi})$ is the composition of the compact operator $\dbarN[0,q][\varphi]$ and a bounded operator, hence it is compact.
 This contradicts $\dim{\bergman(\mathbb C^n,e^{-\varphi})} = \infty$.
\end{proof}

\begin{rem}\label{rem:necessary_condition}
 \begin{inlineenum}
  \item
  To obtain \cref{suffcompgen}, one could also use the fact that compactness of $\dbarN[0,q][\varphi]$ ``percolates up'' the weighted $\dbar$-complex, see \cite{Haslinger2014}: if $1 \leq q \leq n-1$ and $\dbarN[\varphi][0,q]$ is compact, then $\dbarN[\varphi][0,q+1]$ is also compact.
  At the upper end of the $\dbar$-complex, one has
  \begin{align*}
   \clapl[0,n][\varphi] g
    & = \dbar \dbar^t_\varphi g \\
    & = \dbar \bigg(- \psum_{|K|=n-1} \,\sum_{k=1}^n \bigg( \frac{\partial}{\partial z_k} - \frac{\partial \varphi}{\partial z_k} \bigg) g_{kK}\, d\overline z_K \, \bigg) \\
    & = -\frac{1}{4} \Delta g + \sum_{j=1}^n \frac{\partial \varphi}{\partial z_j} \frac{\partial g}{\partial \overline z_j} + \frac{1}{4}(\Delta \varphi) g, \numberthis\label{zeron}
  \end{align*}
  for $g \in \Omega^{0,n}(\mathbb C^n)$, where we identify the $(0,n)$-form
  $$  g= \widetilde g \, d\overline z_1 \wedge \dotsb \wedge d\overline z_n \in \Omega^{0,n}(\mathbb C^n) $$ 
  with its coefficient $\widetilde g \in C^\infty(\mathbb C^n)$.
  If we consider $(0,n)$-forms $g$ with coefficient belonging to $ \bergman(\mathbb C^n, e^{-\varphi}),$
  then $g \in \Omega^{0,n}(\mathbb C^n) \cap L^2_{0,n}(\mathbb C^n,e^{-\varphi})$ and, by \cref{zeron}, we have
  $$ \clapl[0,n][\varphi] g =  \frac{1}{4}(\Delta \varphi) g=  \trace(M_\varphi) g, $$
  and one can now use the same argument as before to reach a contradiction.

  \item
  Let $n=1$ and suppose that $\varphi \colon \mathbb C \to \mathbb R$ is a subharmonic $C^2$ function such that the Bergman space $\bergman(\mathbb C, e^{-\varphi})$ is infinite dimensional.
  The complex Laplacian on $(0,1)$-forms equals $\clapl[0,1][\varphi] = \dbar\dbar_\varphi^*$, and the necessary condition \cref{comptrgen} for compactness of the $\dbar$-Neumann operator becomes
  \begin{equation}\label{eq:rem_necessary_condition_limsup}
   \limsup_{z\to \infty} \Delta \varphi (z)=+\infty.
  \end{equation}
  If one supposes that $  \Delta \varphi $ defines a doubling measure, the condition 
  \begin{equation}\label{eq:rem_necessary_condition_averages_diverging}
   \lim_{z \to \infty} \int_{B_1(z)} \Delta\varphi\,d\lambda = +\infty
  \end{equation}
  is known to be both necessary and sufficient for compactness of the $\ovprt$-Neumann operator, see \cite{Haslinger2007} and \cite{Marzo2009}.
  Of course, \cref{eq:rem_necessary_condition_averages_diverging} implies \cref{eq:rem_necessary_condition_limsup}, as it should.
  
  \item
  A sufficient condition for the Bergman space $\bergman(\mathbb C^n,e^{-\varphi})$ to have infinite dimension is given in \cite[Lemma~3.4]{Shigekawa1991}:
  If $\varphi \colon \mathbb C^n \to \mathbb R$ is a smooth function such that
  \begin{equation}\label{eq:shigekawa_bergman_inf_dim}
   \lim_{z\to \infty} |z|^2 s_1 (z)=+\infty,
  \end{equation}
  where $s_1$ is the smallest eigenvalue of $M_\varphi$, then $\bergman(\mathbb C^n,e^{-\varphi})$ has infinite dimension.
  Condition \cref{eq:shigekawa_bergman_inf_dim} is not sharp: the function
  $$ \varphi (z_1,z_2)= |z_1|^4 + |z_1 z_2|^2 $$
  on $\mathbb C^2$ does not satisfy \cref{eq:shigekawa_bergman_inf_dim}.
  Nevertheless, the corresponding space $\bergman(\mathbb C^2, e^{-\varphi})$ is of infinite dimension since it contains all polynomials in $z_1$, see the computation in \cite[p.~125]{Haslinger2014}.
 \end{inlineenum}
\end{rem}

\section{A stronger result obtained from Schr\"odinger operator theory}\label{sec:schrodinger}

Let $u=u\, d\ovli z_1\wedge \dots \wedge d\ovli z_n$ be a smooth $(0,n)$-form belonging to the domain of $\clapl[0,n][\varphi].$ For $1 \leq j \leq n$ denote by $K_j$ the increasing multiindex $K_j \coloneqq (1,\dotsc,j-1,j+1,\dotsc,n)$ of length $n-1$.\
Then
$$ \dbar_\varphi^t u = \sum_{j=1}^n (-1)^{j+1} \bigg(\frac{\partial \varphi}{\partial z_j}u-\frac{\partial u}{\partial z_j}\bigg)\,d\ovli z_{K_j}.$$
Hence 
\begin{align*}
 \dbar \dbar_\varphi^t u
  &= \Bigg[\sum_{j=1}^n \frac{\partial }{\partial \ovli z_j} \bigg(\frac{\partial \varphi}{\partial z_j}u-\frac{\partial u}{\partial z_j}\bigg)\Bigg] d\ovli z_1 \wedge \dots \wedge d\ovli z_n\\
  &= \Bigg[\sum_{j=1}^n  \bigg(\frac{\partial^2\varphi}{\partial z_j\partial \ovli z_j}u+
      \frac{\partial \varphi}{\partial z_j}
      \frac{\partial u}{\partial \ovli z_j} -
      \frac{\partial^2 u}{\partial z_j\partial \ovli z_j}
      \bigg)\Bigg] d\ovli z_1 \wedge \dots \wedge d\ovli z_n.
\end{align*}
Conjugation with the unitary operator $L^2(\mathbb C^n,e^{-\varphi}) \to L^2(\mathbb C^n)$ of multiplication by $e^{-\varphi/2}$ gives
\begin{equation}\label{eq:complex_laplacian_smooth_n1}
e^{-\varphi/2}  \clapl[0,n][\varphi] e^{\varphi/2}f = \sum_{j=1}^n \bigg(
-\frac{\partial^2 f}{\partial z_j\partial \ovli z_j}-\frac{1}{2} \frac{\partial \varphi}{\partial \ovli z_j}
\frac{\partial f}{\partial z_j} + \frac{1}{2} \frac{\partial \varphi}{\partial z_j}
\frac{\partial f}{\partial \ovli z_j} + \frac{1}{4} \frac{\partial \varphi}{\partial \ovli z_j}
\frac{\partial \varphi}{\partial z_j} f + \frac{1}{2}\frac{\partial^2 \varphi}{\partial z_j\partial \ovli z_j}f \bigg),
\end{equation}
where $f\in L^2(\mathbb C^n)$ and we just wrote down the coefficient of the corresponding $(0,n)$-form.
This operator can be expressed by real variables in the form
\begin{equation}\label{eq:complex_laplacian_smooth_n2}
 e^{-\varphi/2}  \clapl[0,n][\varphi] e^{\varphi/2}f = \frac{1}{4} (-\Delta_A +V)f,
\end{equation}
with
$$ \Delta_A \coloneqq \sum_{j=1}^n \Bigg[\bigg( -\frac{\partial}{\partial x_j} - \frac{i}{2} \frac{\partial \varphi}{\partial y_j} \bigg)^2
+  \bigg( -\frac{\partial}{\partial y_j} + \frac{i}{2} \frac{\partial \varphi}{\partial x_j} \bigg)^2 \Bigg], $$
and $V \coloneqq 2\trace(M_\varphi)$.
It follows that $-\Delta_A +V$ is a Schr\"odinger operator on $L^2(\mathbb R^{2n})$ with magnetic vector potential
$$ A \coloneqq \frac{1}{2} \bigg(-\frac{\partial \varphi}{\partial y_1}, \frac{\partial \varphi}{\partial x_1}, \dotsc, -\frac{\partial \varphi}{\partial y_n}, \frac{\partial \varphi}{\partial x_n}\bigg), $$
where $z_j=x_j+iy_j$, $j=1,\dotsc,n$, and non-negative electric potential $V$ in the case where $\varphi$ is plurisubharmonic.
Note that $A$ corresponds to the real $1$-form $-\frac{i}{2}(\partial - \dbar)\varphi$, and the associated magnetic field is $dA = i\partial\dbar \varphi$.
By essential self-adjointness, the Gaffney extension of $\clapl[0,n][\varphi]$ agrees with the unique self-adjoint extension of the Schr\"odinger operator.
We can now apply a result of Iwatsuka~\cite{Iwatsuka1986} to obtain:

\begin{thm}\label{weighted_problem_discrete_spectrum}
 Let $\varphi \colon \mathbb C^n \to \mathbb R$ be smooth and plurisubharmonic.
 Then $\clapl[0,n][\varphi]$ has compact resolvent on $L^2_{0,n}(\mathbb C^n,e^{-\varphi})$ if and only if the magnetic Schr\"odinger operator $-\Delta_A + V$ has compact resolvent on $L^2(\mathbb R^{2n})$.
 In particular, if there is $1 \leq q \leq n$ such that $\clapl[0,q][\varphi]$ has compact resolvent, then
 \begin{equation}\label{eq:weighted_problem_discrete_spectrum_lim_of_average_square}
  \lim_{z \to \infty} \int_{B(z,1)} \trace(M_\varphi)^2\,d\lambda = +\infty
 \end{equation}
\end{thm}

\begin{proof}
 It remains to prove the last assertion.
 If $\clapl[0,q][\varphi]$ has compact resolvent, then so does $\clapl[0,n][\varphi]$, see \cref{rem:introduction},
 and if the magnetic Schr\"odinger operator has compact resolvent, then
 \begin{equation}\label{eq:weighted_problem_comp_resolv_necessary}
  \int_{B(z,1)} \big(|M_\varphi|^2 + 2\trace(M_\varphi)\big)d\lambda \to \infty
 \end{equation}
 as $z \to \infty$ by \cite[Theorem~5.2]{Iwatsuka1986}, where $|B|^2 \coloneqq \sum_{j,k=1}^n |B_{j,k}|^2$, since the magnetic field is $i\partial\dbar\varphi$.
 Because $M_\varphi \geq 0$, the above is equivalent to $\int_{B(z,1)} \big(\trace(M_\varphi)^2 + \trace(M_\varphi)\big)\,d\lambda \to \infty$ as $z \to \infty$.
 By H\"older's inequality,
 $$ \int_{B(z,1)} \trace(M_\varphi)\,d\lambda \leq |B(z,1)|^{1/2}\bigg(\int_{B(z,1)}\trace(M_\varphi)^2 d\lambda\bigg)^{1/2}, $$
 and we see that \cref{eq:weighted_problem_comp_resolv_necessary} implies $f\big(\int_{B(z,1)}\trace(M_\varphi)^2\,d\lambda \big) \to \infty$ as $z \to \infty$, where $f \colon [0,\infty) \to [0,\infty)$ is $f(t) \coloneqq t + t^{1/2}$.
 Because $f$ is a monotone bijection, this is equivalent to $\int_{B(z,1)}\trace(M_\varphi)^2\,d\lambda \to \infty$ as $z \to \infty$.
\end{proof}

\begin{rem}\label{rem:schrodinger}
 \begin{inlineenum}
  \item
   The above condition \cref{eq:weighted_problem_discrete_spectrum_lim_of_average_square} is not a sufficient condition for compactness of $\dbarN[0,q][\varphi]$ in general, see \cref{rem:decoupled} below.
 
  \item
   If $\varphi \colon \Cplx^n \to \R$ is such that $\trace(M_\varphi)$ satisfies the reverse H\"older condition,
   $$ \bigg(\fint_B \trace(M_\varphi)^r \,d\lambda\bigg)^{1/r} \leq C \fint_B \trace(M_\varphi)\,d\lambda $$
   for some $r \geq 2$, some $C > 0$, and all balls $B \subseteq \Cplx^n$, where $\fint_B$ denotes the average over $B$ for the Lebesgue measure, then H\"older's inequality implies that \cref{eq:weighted_problem_discrete_spectrum_lim_of_average_square} can be replaced by the formally weaker condition of
   \begin{equation}
    \lim_{z \to \infty} \int_{B(z,1)} \trace(M_\varphi)\,d\lambda = +\infty.
   \end{equation}
   The class of functions satisfying one of these reverse H\"older conditions equals $A_\infty \coloneqq \bigcup_{p\geq 1} A_p$, where $A_p$ are the \emph{Muckenhoupt classes}, see \cite[Theorem~3, p.~212]{Stein1993}.
   Every positive polynomial belongs to $A_\infty$.
   In fact, $|P|^a \in A_p$ for $p > 1$ if $-1 < ad < p-1$, where $d$ is the degree of $P$, see \cite[6.5, p~219]{Stein1993}.
 \end{inlineenum}
\end{rem}

\section{Decoupled weights}\label{sec:decoupled_weights}

We will now consider weight functions $\varphi$ which are \emph{decoupled}, meaning that
\begin{equation}\label{eq:decoupled_weights}
 \varphi(z) = \varphi_1(z_1) + \dotsb + \varphi_n(z_n)
\end{equation}
for functions $\varphi_j \colon \mathbb C \to \mathbb R$.
Then
\begin{equation}\label{eq:spectrum_of_decoupled}\
 \spec{\clapl[0,q][\varphi]} = \bigcup_{\substack{q_1+\dotsb+q_n=q \\ q_j \in \{0,1\}}} \big(\spec{\clapl[0,q_1][\varphi_1]} + \dotsb + \spec{\clapl[0,q_n][\varphi_n]}\big)
\end{equation}
and
\begin{equation}\label{eq:essspec_of_decoupled}
 \essspec{\clapl[0,q][\varphi]} = \bigcup_{\substack{q_1+\dotsb+q_n=q \\ q_j \in \{0,1\}}} \; \bigcup_{j=1}^n \bigg(\essspec{\clapl[0,q_j][\varphi_j]} + \sum_{k \neq j} \spec{\clapl[0,q_k][\varphi_k]} \bigg).
\end{equation}
Here, $\clapl[0,q_j][\varphi_j]$ denotes the complex Laplacian for the weight $\varphi_j$ (of one variable) on $\mathbb C$.
This can be seen as a special case of \cite[Theorem~5.5]{Berger2015}.
Indeed, the $\dbar$-operator acting on $L^2_{0,q}(\mathbb C^n,e^{-\varphi})$ can be understood geometrically as the $\dbar^E$-operator for the trivial line bundle $E \coloneqq \mathbb C^n \times \mathbb C \to \mathbb C^n$, and with Hermitian metric $\langle (z,v_1),(z,v_2)\rangle \coloneqq v_1 \ol{v_2}\,e^{-\varphi(z)}$ on $E$.
If now $E_j$ is $\mathbb C \times \mathbb C \to \mathbb C$ with metric given by $\varphi_j \colon \mathbb C \to \mathbb R$, then it is easy to see that under the isomorphism
$$ (\mathbb C^n \times \mathbb C \to \mathbb C) \cong \pi_1^*E_1 \otimes \dotsb \otimes \pi_n^* E_n $$
the trivial line bundle on $\mathbb C^n$ carries the metric given by \cref{eq:decoupled_weights}, where $\pi_j \colon \mathbb C^n \to \mathbb C$ is the projection onto the $j$th factor.
For the details on the general definition of $\dbar^E$ and its properties, we refer to \cite{Wells1973,Ma2007,Huybrechts2005,Demailly2012}.
We note that equation~\cref{eq:spectrum_of_decoupled} has also appeared in \cite{Chakrabarti2010}.

In the following \cref{dbar_neumann_noncompactness_decoupled_weighted_spaces}, we will consider the case where all $\Delta\varphi_j$ define nontrivial doubling measures.
It is known from \cite{Marzo2009} (or \cite[Theorem~2.3]{Haslinger2007}, with slightly stronger assumptions) that, under these conditions, $\dbarN[0,1][\varphi_j]$ (for the one-variable weights) is compact if and only if
\begin{equation}\label{eq:weighted_spaces_condition_compactness_doubling}
 \lim_{z \to \infty} \int_{B_1(z)} \Delta\varphi_j\,d\lambda = +\infty.
\end{equation}
holds.
Using this condition and the above results, we can characterize compactness of the $\dbar$-Neumann operator in terms of the decoupled weight:

\begin{thm}\label{dbar_neumann_noncompactness_decoupled_weighted_spaces}
 Let $\varphi_j \in C^2(\mathbb C,\mathbb R)$ for $1 \leq j \leq n$ with $n \geq 2$, and set $\varphi(z_1,\dotsc,z_n) \coloneqq \varphi_1(z_1) + \dotsb + \varphi_n(z_n)$.
 Assume that all $\varphi_j$ are subharmonic and such that $\Delta\varphi_j$ defines a nontrivial doubling measure, then
 \begin{enumerate}
  \item\label{item:dbar_neumann_noncompactness_decoupled_weighted_spaces_bergman_space_dim}
   $\dim{\ker(\clapl[0,0][\varphi])} = \dim{\bergman(\mathbb C^n,e^{-\varphi})} = \infty$,
  \item\label{item:dbar_neumann_noncompactness_decoupled_weighted_spaces_trivial_kernel}
   $\ker(\clapl[0,q][\varphi]) = 0$ for $q \geq 1$,
  \item\label{item:dbar_neumann_noncompactness_decoupled_weighted_spaces_bounded}
   $\dbarN[0,q][\varphi]$ is bounded for $0 \leq q \leq n$,
  \item\label{item:dbar_neumann_noncompactness_decoupled_weighted_spaces_bot}
   $\dbarN[0,q][\varphi]$ with $0 \leq q \leq n-1$ is not compact, and
  \item\label{item:dbar_neumann_noncompactness_decoupled_weighted_spaces_top}
   $\dbarN[0,n][\varphi]$ is compact if and only if
   \begin{equation}\label{eq:dbar_neumann_compactness_char_decoupled_top}
    \lim_{z \to \infty} \int_{B_1(z)} \trace(M_\varphi)\,d\lambda = \infty.
   \end{equation}
 \end{enumerate}
\end{thm}

\begin{proof}
 From \cite[Theorem~C]{Marco2003}, it follows from our assumptions on $\varphi$ that
 $\dbar$ has closed range in $L^2_{0,1}(\mathbb C,e^{-\varphi_j})$ for all $j$.
 From this it also follows that $\dbar$ has closed range in $L^2_{0,q}(\mathbb C^n,e^{-\varphi})$ for all $0 \leq q \leq n$, see \cite[Theorem~4.5]{Chakrabarti2011} or \cite[Corollary~2.15]{Bruening1992}, which implies that the $\dbar$-Neumann operator is bounded.
 
 Moreover, $\ker(\clapl[0,1][\varphi_j]) = 0$ for all $j$.
 In fact, in this simple case, we obtain from the Kohn-Morrey-H\"ormander formula \cref{eq:kohn_morrey}
 \begin{equation}\label{eq:bochner_kodaira_nakano_weighted}
  \|\clapl[0,1][\varphi_j] u\|^2 = \|\dbar_{\varphi_j}^* u\|^2 \geq \frac{1}{4} \int_{\mathbb C} \Delta\varphi_j |u_1|^2 \, e^{-\varphi_j}d\lambda
 \end{equation}
 for all forms $u = u_1 d\overline z \in \dom{\dbar_{\varphi_j}^*}$.
 If $u \in \ker(\clapl[0,1][\varphi_j])$, then $u$ is smooth by elliptic regularity, and if $z_0$ is such that $\Delta\varphi_j(z_0) > 0$ (note that $\Delta\varphi_j \geq 0$ everywhere by subharmonicity), then $u = 0$ in a neighborhood of $z_0$.
 But then $u = 0$ everywhere since $\mathbb C$ is connected and by using a unique continuation principle of Aronszajn, see \cite{Aronszajn1957} or \cite[p.~333]{Demailly2012}.
 We also have $\ker(\clapl[0,q][\varphi])=0$ for all $q \geq 1$, either by combining the above with the \emph{K\"unneth formula}, 
 \begin{equation}\label{eq:kuenneth}
  \ker(\clapl[0,q][\varphi]) \cong \bigoplus_{q_1 + \dotsb + q_n = q} \ker(\clapl[0,q_1][\varphi_1]) \htensor \dotsb \htensor \ker(\clapl[0,q_n][\varphi_n]),
 \end{equation}
 see \cite[Corollary~2.15]{Bruening1992} or \cite[Theorem~4.5]{Chakrabarti2011}, where $\htensor$ denotes the Hilbert space tensor product, or directly by using the same argument as above, i.e., applying \cref{eq:kohn_morrey} to the higher dimensional problem and using that $M_\varphi \neq 0$ at one point.
 
 As a nontrivial doubling measure, $\Delta\varphi_j$ satisfies $\int_{\mathbb C} \Delta\varphi_j\,d\lambda = \infty$.
 Consequently, the weighted Bergman space $\bergman(\mathbb C,e^{-\varphi_j})$ has infinite dimension by \cite[Theorem~3.2]{Rozenblum2006}.
 This also implies, by \cref{eq:kuenneth}, that $\dim{\ker(\clapl[0,0][\varphi])} = \infty$.
 On the other hand, $\dim{\bergman(\mathbb C,e^{-\varphi_j})} = \infty$ for at least one $1 \leq j \leq n$ implies that $\dbarN[0,q][\varphi]$ cannot be compact for $0 \leq q \leq n-1$, see \cite[Theorem~5.6]{Berger2015}:
 indeed, $0 \in \essspec{\clapl[0,0][\varphi_j]}$ and hence \cref{eq:essspec_of_decoupled} implies that the infinite sets $\spec{\clapl[0,0][\varphi_k]}$ for $j \neq k$ are contained in $\essspec{\clapl[0,q][\varphi]}$.
 This finishes the proof of \cref{%
  item:dbar_neumann_noncompactness_decoupled_weighted_spaces_trivial_kernel,%
  item:dbar_neumann_noncompactness_decoupled_weighted_spaces_bergman_space_dim,%
  item:dbar_neumann_noncompactness_decoupled_weighted_spaces_bounded,%
  item:dbar_neumann_noncompactness_decoupled_weighted_spaces_bot%
 }.
 Again by \cite[Theorem~5.6]{Berger2015}, $\dbarN[0,n][\varphi]$ is compact if and only if all $\dbarN[0,1][\varphi_j]$ are compact, which is the case if and only if \cref{eq:weighted_spaces_condition_compactness_doubling} holds for all $1 \leq j \leq n$.
 It remains to show that this is equivalent to \cref{eq:dbar_neumann_compactness_char_decoupled_top}.
 By a simple scaling argument, the claim is equivalent to
 \begin{equation}\label{eq:integrals_polydisk}
  \int_{B_1(z_1) \times \dotsb \times B_n(z_n)} \trace(M_\varphi)\,d\lambda = \frac{\pi^{n-1}}{4} \sum_{j=1}^n \int_{B_1(z_j)} \Delta \varphi_j\,d\lambda \to \infty \quad\text{as}\quad z = (z_1,\dotsc,z_n) \to \infty,
 \end{equation}
 and if \cref{eq:weighted_spaces_condition_compactness_doubling} holds for all $1 \leq j \leq n$, then \cref{eq:integrals_polydisk} is also satisfied.
 Conversely, if \cref{eq:integrals_polydisk} is true, then choosing $z = \zeta e_k$ with $\zeta \in \mathbb C$ and $e_k$ the $k$th standard basis vector of $\mathbb C^n$ implies
 $$ \int_{B_1(\zeta)}\Delta \varphi_k\,d\lambda + \sum_{j \neq k} \int_{B_1(0)}\Delta \varphi_j\,d\lambda \to \infty \quad\text{as}\quad \zeta \to \infty, $$
 so that $\lim_{\zeta \to \infty}\int_{B_1(\zeta)}\Delta\varphi_k \,d\lambda = \infty$ since the second term is bounded.
 This shows \cref{item:dbar_neumann_noncompactness_decoupled_weighted_spaces_top} and concludes the proof.
\end{proof}

\begin{rem}\label{rem:decoupled}
 \begin{inlineenum}
  \item
  The doubling condition is satisfied if the $\Delta \varphi_j$ belong to $A_\infty$, see \cite[p.~196]{Stein1993}, where we recall from \cref{rem:schrodinger} that $A_\infty$ is the union of the Muckenhoupt classes.
  As an example, $z \mapsto |z|^\alpha$ is in $A_p$ for $p > 1$ if and only if $-2 < \alpha < 2(p-1)$, and defines a doubling measure for $-2 < \alpha$, cf.\ \cite[6.4,~p.~218]{Stein1993}.
  Since $\Delta |z|^\alpha = \alpha^2 |z|^{\alpha-2}$, we see that $\varphi_j(z) = |z|^\alpha$ satisfies the assumptions of \cref{dbar_neumann_noncompactness_decoupled_weighted_spaces} for all $\alpha \geq 4$ (so that $\varphi$ is at least $C^2$).

  \item
  Let $n\ge 2$ and consider the weight function
  \begin{equation}\label{eq:decoupled_sum_of_powers}
   \varphi (z) = \sum_{j=1}^n |z_j|^{\alpha_j},
  \end{equation}
  where $\alpha_j \in \mathbb R$, $\alpha_j \geq 4$.
  Then $\varphi \in C^2(\mathbb C^n,\mathbb R)$ and by the above
  $$ \lim_{z\to \infty} \trace(M_\varphi (z)) = \lim_{z \to \infty} \frac{1}{4}\sum_{j=1}^n \alpha_j^2 |z_j|^{\alpha_j-2} = +\infty. $$
  Therefore it follows from \cref{dbar_neumann_noncompactness_decoupled_weighted_spaces} that
  $\dbarN[0,q][\varphi]$ with $0 \leq q \leq n-1$ is not compact while $\dbarN[0,n][\varphi]$ is compact.
  Hence, the necessary conditions \cref{comptrgen} and \cref{eq:weighted_problem_discrete_spectrum_lim_of_average_square} fail to be sufficient for compactness of $\dbarN[0,q][\varphi]$ for $0\le q \le n-1$ in general.
  Of course, by \cref{rem:schrodinger,dbar_neumann_noncompactness_decoupled_weighted_spaces}, the integral condition \cref{eq:weighted_problem_discrete_spectrum_lim_of_average_square} is both necessary and sufficient for compactness of $\dbarN[0,n][\varphi]$ for plurisubharmonic weights $\varphi$ with $\trace(M_\varphi) \in A_\infty$, such as \cref{eq:decoupled_sum_of_powers}.

  \item
  Using a variation of the above decoupled weights, one easily sees that, for $q > n/2$, there is a plurisubharmonic function $\varphi_q \colon \mathbb C^n \to \mathbb R$, such that $\dbarN[0,k][\varphi_q]$ is compact precisely for $q \leq k \leq n$.
  Indeed, one may take
  $$ \varphi_q(z_1,\dotsc,z_n) \coloneqq |(z_1,\dotsc,z_{q-1})|^4 + |(z_q,\dotsc,z_n)|^4. $$
  Then both of the spaces $\bergman(\Cplx^{q-1},e^{-\varphi_1})$ and $\bergman(\Cplx^{n-q+1},e^{-\varphi_2})$, where $\varphi_1 \colon \Cplx^{q-1}\to\R$, $\varphi(z) \coloneqq |z|^4$, and $\varphi_2 \colon \Cplx^{n-q+1}\to\R$, $\varphi_2(z) \coloneqq |z|^4$, have infinite dimension by the result of Shigekawa quoted in \cref{rem:necessary_condition}.
  Moreover, the $\dbar$-Neumann operators $\dbarN[0,q][\varphi_j]$ are compact for $q \geq 1$ and $j \in \{1,2\}$, as is easily deduced by verifying \cref{sq}.
  Since $n-q+1 < n/2+1$ implies $n-q+1 \leq n/2 \leq q-1$, one obtains from \cite[Theorem~5.2]{Berger2015} that $\dbarN[0,k][\varphi]$ is compact exactly for $k = q-1 + j$ with $j \geq 1$, as claimed.
 \end{inlineenum}
\end{rem}

\bibliography{./references}

\begin{thebibliography}{24}
\providecommand{\natexlab}[1]{#1}
\providecommand{\url}[1]{\texttt{#1}}
\expandafter\ifx\csname urlstyle\endcsname\relax
  \providecommand{\doi}[1]{doi: #1}\else
  \providecommand{\doi}{doi: \begingroup \urlstyle{rm}\Url}\fi

\bibitem[Aronszajn(1957)]{Aronszajn1957}
N.~Aronszajn.
\newblock A unique continuation theorem for solutions of elliptic partial
  differential equations or inequalities of second order.
\newblock \emph{J. Math. Pures Appl. (9)}, 36:\penalty0 235--249, 1957.
\newblock ISSN 0021-7824.

\bibitem[Berger(2016)]{Berger2015}
F.~Berger.
\newblock {Essential spectra of tensor product Hilbert complexes, and the
  $\overline\partial$-Neumann problem on product manifolds}.
\newblock \emph{Journal of Functional Analysis}, 2016.
\newblock \doi{10.1016/j.jfa.2016.06.004}.

\bibitem[Br\"{u}ning and Lesch(1992)]{Bruening1992}
J.~Br\"{u}ning and M.~Lesch.
\newblock {Hilbert complexes}.
\newblock \emph{J. Funct. Anal.}, 108\penalty0 (1):\penalty0 88--132, 1992.
\newblock ISSN 0022-1236.
\newblock \doi{10.1016/0022-1236(92)90147-B}.

\bibitem[Chakrabarti(2010)]{Chakrabarti2010}
D.~Chakrabarti.
\newblock {Spectrum of the complex {L}aplacian on product domains}.
\newblock \emph{Proc. Amer. Math. Soc.}, 138\penalty0 (9):\penalty0 3187--3202,
  2010.
\newblock ISSN 0002-9939.
\newblock \doi{10.1090/S0002-9939-10-10522-X}.

\bibitem[Chakrabarti and Shaw(2011)]{Chakrabarti2011}
D.~Chakrabarti and M.-C. Shaw.
\newblock The {C}auchy-{R}iemann equations on product domains.
\newblock \emph{Math. Ann.}, 349\penalty0 (4):\penalty0 977--998, 2011.
\newblock ISSN 0025-5831.
\newblock \doi{10.1007/s00208-010-0547-x}.

\bibitem[Dall'Ara(2014)]{DallAra2014}
G.~M. Dall'Ara.
\newblock \emph{{Matrix Schr\"odinger operators and weighted Bergman kernels}}.
\newblock PhD thesis, 2014.

\bibitem[Dall'Ara(2015)]{DallAra2015}
G.~M. Dall'Ara.
\newblock {Coercivity of weighted Kohn Laplacians: the case of model monomial
  weights in $\mathbb C^2$}, 2015.

\bibitem[Demailly(2012)]{Demailly2012}
J.-P. Demailly.
\newblock \emph{{Complex analytic and differential geometry}}.
\newblock 2012.
\newblock Available online at
  \url{https://www-fourier.ujf-grenoble.fr/~demailly/}.

\bibitem[Gaffney(1955)]{Gaffney1955}
M.~P. Gaffney.
\newblock {Hilbert space methods in the theory of harmonic integrals}.
\newblock \emph{Trans. Amer. Math. Soc.}, 78:\penalty0 426--444, 1955.
\newblock ISSN 0002-9947.

\bibitem[Haslinger(2011)]{Haslinger2011}
F.~Haslinger.
\newblock Compactness for the {$\overline\partial$}-{N}eumann problem: a
  functional analysis approach.
\newblock \emph{Collect. Math.}, 62\penalty0 (2):\penalty0 121--129, 2011.
\newblock ISSN 0010-0757.
\newblock \doi{10.1007/s13348-010-0013-9}.

\bibitem[Haslinger(2014)]{Haslinger2014}
F.~Haslinger.
\newblock \emph{{The {$\overline{\partial}$}-{N}eumann problem and
  {S}chr\"{o}dinger operators}}, volume~59 of \emph{{De Gruyter Expositions in
  Mathematics}}.
\newblock De Gruyter, Berlin, 2014.
\newblock ISBN 978-3-11-031530-1; 978-3-11-031535-6.
\newblock \doi{10.1515/9783110315356}.

\bibitem[Haslinger and Helffer(2007)]{Haslinger2007}
F.~Haslinger and B.~Helffer.
\newblock {Compactness of the solution operator to {$\overline\partial$} in
  weighted {$L^2$}-spaces}.
\newblock \emph{J. Funct. Anal.}, 243\penalty0 (2):\penalty0 679--697, 2007.
\newblock ISSN 0022-1236.
\newblock \doi{10.1016/j.jfa.2006.09.004}.

\bibitem[H{\"o}rmander(1965)]{Hoermander1965}
L.~H{\"o}rmander.
\newblock {$L^{2}$} estimates and existence theorems for the {$\bar \partial
  $}\ operator.
\newblock \emph{Acta Math.}, 113:\penalty0 89--152, 1965.
\newblock ISSN 0001-5962.

\bibitem[Huybrechts(2005)]{Huybrechts2005}
D.~Huybrechts.
\newblock \emph{{Complex geometry}}.
\newblock {Universitext}. Springer-Verlag, Berlin, 2005.
\newblock ISBN 3-540-21290-6.
\newblock An introduction.

\bibitem[Iwatsuka(1986)]{Iwatsuka1986}
A.~Iwatsuka.
\newblock Magnetic {S}chr\"odinger operators with compact resolvent.
\newblock \emph{J. Math. Kyoto Univ.}, 26\penalty0 (3):\penalty0 357--374,
  1986.
\newblock ISSN 0023-608X.

\bibitem[Ma and Marinescu(2007)]{Ma2007}
X.~Ma and G.~Marinescu.
\newblock \emph{{Holomorphic {M}orse inequalities and {B}ergman kernels}},
  volume 254 of \emph{{Progress in Mathematics}}.
\newblock Birkh\"{a}user Verlag, Basel, 2007.
\newblock ISBN 978-3-7643-8096-0.

\bibitem[Marco et~al.(2003)Marco, Massaneda, and Ortega-Cerd\`{a}]{Marco2003}
N.~Marco, X.~Massaneda, and J.~Ortega-Cerd\`{a}.
\newblock {Interpolating and sampling sequences for entire functions}.
\newblock \emph{Geom. Funct. Anal.}, 13\penalty0 (4):\penalty0 862--914, 2003.
\newblock ISSN 1016-443X.
\newblock \doi{10.1007/s00039-003-0434-7}.

\bibitem[Marzo and Ortega-Cerd\`{a}(2009)]{Marzo2009}
J.~Marzo and J.~Ortega-Cerd\`{a}.
\newblock {Pointwise estimates for the {B}ergman kernel of the weighted {F}ock
  space}.
\newblock \emph{J. Geom. Anal.}, 19\penalty0 (4):\penalty0 890--910, 2009.
\newblock ISSN 1050-6926.
\newblock \doi{10.1007/s12220-009-9083-x}.

\bibitem[Rozenblum and Shirokov(2006)]{Rozenblum2006}
G.~Rozenblum and N.~Shirokov.
\newblock {Infiniteness of zero modes for the {P}auli operator with singular
  magnetic field}.
\newblock \emph{J. Funct. Anal.}, 233\penalty0 (1):\penalty0 135--172, 2006.
\newblock ISSN 0022-1236.
\newblock \doi{10.1016/j.jfa.2005.08.001}.

\bibitem[Ruppenthal(2011)]{Ruppenthal2011a}
J.~Ruppenthal.
\newblock Compactness of the {$\overline\partial$}-{N}eumann operator on
  singular complex spaces.
\newblock \emph{J. Funct. Anal.}, 260\penalty0 (11):\penalty0 3363--3403, 2011.
\newblock ISSN 0022-1236.
\newblock \doi{10.1016/j.jfa.2010.12.022}.

\bibitem[Shigekawa(1991)]{Shigekawa1991}
I.~Shigekawa.
\newblock Spectral properties of {S}chr\"odinger operators with magnetic fields
  for a spin {$\frac12$} particle.
\newblock \emph{J. Funct. Anal.}, 101\penalty0 (2):\penalty0 255--285, 1991.
\newblock ISSN 0022-1236.
\newblock \doi{10.1016/0022-1236(91)90158-2}.

\bibitem[Stein(1993)]{Stein1993}
E.~M. Stein.
\newblock \emph{{Harmonic analysis: real-variable methods, orthogonality, and
  oscillatory integrals}}, volume~43 of \emph{{Princeton Mathematical Series}}.
\newblock Princeton University Press, Princeton, NJ, 1993.
\newblock ISBN 0-691-03216-5.
\newblock With the assistance of Timothy S. Murphy, Monographs in Harmonic
  Analysis, III.

\bibitem[Straube(2010)]{Straube2010}
E.~J. Straube.
\newblock \emph{Lectures on the {$\mathscr{L}^2$}-{S}obolev theory of the
  {$\overline{\partial}$}-{N}eumann problem}.
\newblock ESI Lectures in Mathematics and Physics. European Mathematical
  Society (EMS), Z\"urich, 2010.
\newblock ISBN 978-3-03719-076-0.
\newblock \doi{10.4171/076}.

\bibitem[Wells(1973)]{Wells1973}
R.~O. Wells, Jr.
\newblock \emph{{Differential analysis on complex manifolds}}.
\newblock {Prentice-Hall, Inc., Englewood Cliffs, N.J.}, 1973.
\newblock Prentice-Hall Series in Modern Analysis.

\end{thebibliography}

\end{document}